\documentclass[12pt,a4paper]{amsart}

\usepackage[latin1]{inputenc}
\usepackage[spanish,english]{babel}
\usepackage{graphicx}

\usepackage{tabularx}

\usepackage{amsmath}
\usepackage{amsthm}
\usepackage{amsmath, amsthm, amscd, amsfonts, amssymb, graphicx, color}

\usepackage{pifont}

\usepackage{amsfonts}
\usepackage{graphicx}

\newtheorem{exer}{\sffamily\bfseries Ejercicio}
\newtheorem{teo}{Theorem}
\newtheorem{lem}[teo]{Lemma}

\newtheorem{prop}[teo]{Proposition}

\newtheorem{rem}[teo]{Remark}

\numberwithin{equation}{section}

\newcommand{\kah}{\mathcal{K}(\mathcal{H})^{ah}}
\newcommand{\kh}{\mathcal{K}(\mathcal{H})^{h}}
\newcommand{\uc}{\mathcal{U}_c(\mathcal{H})}
\newcommand{\oa}{\mathcal{O}_A}
\newcommand{\uu}{\mathcal{U}}
\newcommand{\aaa}{\mathcal{A}}
\newcommand{\bb}{\mathcal{B}}
\newcommand{\bah}{\mathcal{B}(\mathcal{H})^{ah}}
\newcommand{\bh}{\mathcal{B}(\mathcal{H})}
\newcommand{\kk}{\mathcal{K}}
\newcommand{\PP}{\mathcal{P}}

\newcommand{\h}{\mathcal{H}}

\newcommand{\N}{\mathbb N}
\newcommand{\C}{\mathbb C}

\newcommand{\R}{\mathbb R}

\newcommand{\I}{\mathcal I}

\newcommand{\longi}{{\rm L}}
\newcommand{\dist}{{\rm dist}}
\newcommand{\bit}{\begin{itemize}}
\newcommand{\eit}{\end{itemize}}
\newcommand{\be}{\begin{enumerate}}
\newcommand{\ee}{\end{enumerate}}
\newcommand{\bx}[1]{\begin{exer}\rm{#1}}
\newcommand{\ex}{\end{exer}}

\newcommand{\ba}{\begin{array}}
\newcommand{\ea}{\end{array}}
\newcommand{\bc}{\begin{center}}
\newcommand{\ec}{\end{center}}

\newcommand{\g}{\gamma}
\newcommand{\s}{\beta}
\newcommand{\de}{\delta}

\newcommand{\D}{\mathcal{D}}

\newcommand{\bq}{\begin{equation}}
\newcommand{\eq}{\end{equation}}

\begin{document}

\title{\vspace*{0cm}\textbf{Minimal length curves in unitary orbits of a Hermitian compact operator} }
%
%

\author{Tamara Bottazzi $^1$ and Alejandro Varela$^{1,2}$}
\address{$^1$ Instituto Argentino
de Matem\'atica ``Alberto P. Calder\'on'', Saavedra 15 3? piso,
(C1083ACA) Ciudad Aut\'onoma de Buenos Aires, Argentina} 

\address{$^2$ Instituto de Ciencias, Universidad Nacional de General Sarmiento, J.
M. Gutierrez 1150, (B1613GSX) Los Polvorines, Pcia. de Buenos Aires, Argentina} 
\email{tpbottaz@ungs.edu.ar, avarela@ungs.edu.ar}

\subjclass[2010]{MSC: Primary: 22F30, 43A85, 47B15, 47A58, 53C22. Secondary: 47B07, 47B10, 47C15.}
\keywords{
Unitary orbits, geodesic curves, minimal operators in quotient spaces, approximation of minimal length curves.}

\begin{abstract} 
We study some examples of minimal length curves in homogeneous spaces of $\bb(\h)$ under a left action of a unitary group. Recent results relate these curves with the existence of minimal (with respect to a quotient norm) anti-Hermitian operators $Z$ in the tangent space of the starting point. We show minimal curves that are not of this type but nevertheless can be approximated uniformly by those.
\end{abstract}

\maketitle

\section{Introduction} \label{intro}
Let $\h$ be a separable Hilbert space and $\kk(\h)$ be the algebra of compact operators. In this work we consider the orbit manifold of a self-adjoint compact operator $A$ by a particular unitary group, that is
$$
\mathcal{O}_A=\{uAu^*: u \text{ unitary in } \bb(\h) \text{ and } u-1\in\mathcal{K}(\h)\}.
$$
Given two points, $x,y\in \oa$, the rectifiable distance between them is the infimum of the lengths of all the smooth curves in $\oa$ that join $x$ and $y$. Our purpose is to study the existence and properties of some particular minimal length curves in $\oa$.

The tangent space at any $b\in\mathcal{O}_A$ is 
$$(T\oa)_b=\{zb-bz: z\in \kk(\h),\ z^*=-z \}$$
endowed with the Finsler metric given by the usual operator norm $\left\|\cdot\right\|$. If $x\in (T\oa)_b$, the existence of a (not necessarily unique) minimal element $z_0$ such that
$$\left\|x\right\|_b=\|z_0\|=\inf \left\{\left\|z\right\|:z\in \kk(\h),\ z^*=-z, \ zb-bz=x \right\}$$
allows in \cite{andruchow larotonda} the description of minimal length curves of the manifold by the parametrization
$$\gamma(t)=e^{tz_0}\  b\  e^{-tz_0} ,\ t\in\left[-\frac{\pi}{2\left\|z_0\right\|},\frac{\pi}{2\left\|z_0\right\|}\right].$$
These $z_0$ can be described as $i(C+D)$, with $C\in\kk(\h)$, $C^*=C$ and $D$ a real diagonal operator in an orthonormal basis of eigenvectors of $A$.

If we consider $\bb\subset\aaa$ von Neumann algebras and $a\in\mathcal{A}$, $a^*=a$, there always exists an element $b_0$ in $\mathcal{B}$ such that $\left\|a+b_0\right\|\leq \left\|a+b\right\|$, for all $b\in \mathcal{B}$ (see \cite{dmr1}). The element $a_0+b_0$ is called minimal in the class $[a_0]$ of $\aaa^h/\bb^h$. However, in the case of $\aaa=\kk(\h)$, a $C^*$-algebra which is not a von Neumann algebra, and $\bb\subset \kk(\h)$ a subalgebra there is not always a minimal compact operator in any class in $\kk(\h^h)/\bb^h$. In \cite{bv} we exhibit an example of this fact. In this case, the existence of a best approximant for $C\in \kk(\h)$, $C^*=C$ is guaranteed when $C$, for example, has finite rank (see Proposition 5.1 in\cite{andruchow larotonda}). 

The above motivated us to study the following, among other issues, in the unitary orbit of a Hermitian operator. Let $b\in \oa$ and $x\in (T\oa)_b$ and suppose that there exists a uniparametric curve $\psi(t)=e^{tZ}be^{-tZ}$ which is a minimal length curve among all the smooth curves joining $b$ and $\psi(t)$ in $\oa$ for $t\in\left[-\frac{\pi}{2\left\|Z\right\|},\frac{\pi}{2\left\|Z\right\|}\right]$:
\bit
\item  Would $Z$ be a compact minimal lifting of $x$ (i.e $x=Zb-bZ$ and $\left\|Z\right\|=\left\|x\right\|_b$)?
\item Can $\psi$ be approximated in $\oa$ by a sequence of minimal length curves of matrices?
\eit
The present work continues the analysis made in \cite{andruchow larotonda} of this homogeneous spaces and we use minimality characterizations that we developed in \cite{bv}.

The results in this paper are divided in three parts. In the first we describe and study minimal length curves in the orbit of a particular compact Hermitian operator. In the second part we construct a sequence of minimal length curves of matrices which converges uniformly to the minimal length curves found in the first part. Finally, in the third part we study cases of anti-Hermitian compact operators such that their best bounded diagonal approximants are not compact and the properties of the minimal curves they determinate.

\section{Preliminaries and notation} \label{preliminares}

Let $(\h,\left\langle ,\right\rangle)$ be a separable Hilbert space. We denote by $\left\|h\right\|=\left\langle h,h \right\rangle^{1/2}$ the norm for each $h\in\h$.
Let $\bb(\h)$ denote the set of bounded operators (with the identity operator $I$) and $\kk(\h)$, the two-sided closed ideal of compact operators on $\h$.
Given $\aaa\subset\bb(\h)$, we use the superscript $^{ah}$  (resp. $^{h}$) to note the subset of anti-Hermitian (resp. Hermitian) elements of $\aaa$.

We consider the group of unitary operators in $\bb(\h)$
$$\uu(\h)=\{u\in \bb(\h):\ uu^*=u^*u=I\}$$
and the unitary Fredholm group, defined as
$$\uc=\{u\in \mathcal{U}(\h): u-I\in \kk(\h)\}.$$

We denote with $\left\|\cdot\right\|$ the usual operator norm in $\bb(\h)$ and with $[\ ,\ ]$ the commutator operator, that is, for any $T,S\in \bb(\h)$
$$[T,S]=TS-ST.$$
It should be clear from the context the use of the same notation $\left\|\cdot\right\|$ to refer to the operator norm or the norm on $\h$.

We define the unitary orbit of a fixed $A\in \kk(\h)$, $A=A^*$, as
\begin{equation}
\mathcal{O}_A=\{uAu^*:u\in\mathcal{U}_c(\h)\}\ \subset\ \kk(\h).\label{orbita}
\end{equation}
$\oa$ is an homogeneous space if we consider the action $\pi_b:\uu_c(\h)\to\oa$, $\pi_b(u)=ubu^*$. For each $b\in \oa$, the isotropy group $\I_b$ is
$$\I_b=\{u\in \uc:\ ubu^*=b\}.$$
Since for each $u\in \uc$ there always exists $X\in \kah$ such that $u=e^X$ (see Proposition \ref{caracterizacion de unitarios}), the isotropy can be redefined by
$$\I_b=\{e^X\in \uc:\ X\in \kah,\ [X,b]=0 \}.$$
For each $b\in \oa$, its tangent space is
$$(T\oa)_b=\{Yb-bY: Y\in\kah \}\subset \kah.$$
Consider a smooth curve (i.e. $C^1$ and with derivative non equal to zero) $u:[0,1]\to \uc$  such that $u(0)=1$ y $u'(0)=Y$, then the differential of the surjective map $\pi_b$ at $1$ is
$$(d\pi_b)_1(Y)=\frac{d}{dt}\left.\pi_b(u(t))\right|_{t=0}=u'(0)b\ u^*(0)+u(0)b\ u'(0)^*$$
$$=Yb1^*+1bY^*=Yb-bY=[Y,b].$$
For every $b\in \oa$ we consider each tangent space as 
$$(T\oa)_b\cong(T\uc)_1/(T\I_b)1\cong\kah/(\{b\}')^{ah},$$
being $\{b\}'$ the set of elements that commute with $b$ in a $C^*$-algebra $\aaa$ (in this particular case $\aaa=\kk(\h)$). 
Let us consider the Finsler metric, defined for each $x\in (T\oa)_b$ as
$$\|x\|_b=\inf\{\|Y\|: Y\in\kah \hbox{ such that } [Y,b]=x\}$$
This metric can be expressed in terms of the projection to the quotient $\kah/(\{b\}')^{ah}$ as
$$\left\|Yb-bY\right\|_b=\left\|\left[Y\right]\right\|=\inf_{C\in (\{b\}')^{ah}}\ \left\|Y+C\right\|$$
for each class $[Y]=\left\{Y+C:\ C\in (\{b\}')^{ah}\right\}$. This Finsler norm is invariant under the action of $\uc$. 

There always exists $Z\in \bb(\h)^{ah}$ such that $[Z,b]=x$ and $\left\|Z\right\|=\left\|x\right\|_b$. Such element $Z$ is called minimal lifting for $x$, and $Z$ may not be compact and/or unique (see \cite{bv}).
Consider piecewise smooth curves $\beta:[a,b]\to \oa$. We define the rectifiable length of $\beta$ as 
$$\longi(\beta)=\int_a^b\left\|\beta'(t)\right\|_{\beta(t)}\ dt,$$
and the rectifiable distance between two points of $\oa$, named $c_1$ and $c_2$, as
$$\dist(c_1,c_2)=\inf\{\longi(\beta):\beta\ {\rm is\ smooth},\beta(a)=c_1,\beta(b)=c_2\}.$$
If $\aaa$ is any $C^*$-algebra of $\bh$ and $\left\{e_k\right\}_{k=1}^{\infty}$ is a fixed orthonormal basis of $\h$, we denote with $\D(\aaa)$ the set of diagonal operators with respect to this basis, that is
$$\D(\aaa)=\left\{T\in \aaa:\ \left\langle Te_i,e_j\right\rangle=0\ ,\ \text{ for all } i\neq j\right\}.$$
Given an operator $Z\in \aaa$, if there exists an operator $D_1\in \D(\aaa)$ such that
$$\left\|Z+D_1\right\|={\rm dist}\left(Z,\D\left(\aaa\right)\right),$$
we say that $D_1$ is a best approximant of $Z$ in $\D(\aaa)$. In other terms, the operator $Z+D_1$ verifies the following inequality
$$\left\|Z+D_1\right\|\leq \left\|Z+D\right\|$$
for all $D\in \D(\aaa)$. In this sense, we call $Z+D_1$ a minimal operator or similarly we say that $D_1$ is minimal for $Z$. If $Z$ is anti-Hermitian it holds that 
$${\rm dist}\left(Z,\D\left(\aaa\right)\right)={\rm dist}\left(Z,\D\left(\aaa^{ah}\right)\right),$$
since $\left\|Im(X)\right\|\leq \left\|X\right\|$ for every $X\in \aaa$.

Let $T\in \bb(\h)$ and consider the coefficients $T_{ij}=\left\langle Te_i,e_j\right\rangle$ for each $i,j\in \N$, that define an infinite matrix $\left(T_{ij}\right)_{i,j\in \N}$. The $j$th-column and $i$th-row of $T$ are the vectors in $\ell^2$ given by $c_j(T)=\left(T_{1j},T_{2j},...\right)$ and $f_j(T)=\left(T_{i1},T_{i2},...\right)$, respectively.

We use $\sigma(T)$ and $R(T)$ to denote the spectrum and range of $T\in \bh$, respectively. 

We define $\Phi:\bb(\h)\to \D(\bb(\h))$, $\Phi(X)= {\rm Diag}(X)$, that takes the main diagonal (i.e the elements of the form $\{\left\langle Xe_i,e_i\right\rangle\}_{i\in\N}$) of an operator $X$ and builds a diagonal operator in the chosen fixed basis of $\h$. For a given bounded sequence $\{d_n\}_{n\in\N}\subset \C$ we denote with ${\rm Diag}\big( \{d_n\}_{n\in\N}\big)$ the diagonal (infinite) matrix with $\{d_n\}_{n\in\N}$ in its diagonal  and $0$ elsewhere.

The following theorem is similar to Theorem 1 in \cite{bv} but this version only requires that $T\in \bh^h$, instead of $T\in \kh$. The proof is exactly the same tha the one in the case where $T$ is compact. 

\begin{teo} \label{teo minimal}
Let $T\in \bh^h$ described as an infinite matrix by $\left(T_{ij}\right)_{i,j\in \N}$.
Suppose that $T$ satisfies:
\be
\item $T_{ij}\in\R$ for each $i,j\in \N$,
\item there exists $i_0\in \N$ satisfying $T_{i_0 i_0}=0$, with $T_{i_0 n}\neq 0$, for all $n\neq i_0$,
\item if $T^{[i_0]}$ is the operator $T$ with zero in its $i_0$th-column and $i_0$th-row then $$
\left\|c_{i_0}(T)\right\|\geq \left\|T^{[i_0]}\right\|
$$
(where $\left\|c_{i_0}(T)\right\|$ denotes the Hilbert norm of the $i_0$th-column of $T$), and
\item if the $T_{nn}$'s satisfy that, for each $n\in\N$, $n\neq i_0$:
$$
T_{nn}=-\dfrac{\left\langle c_{i_0}(T),c_n(T)\right\rangle}{T_{i_0n}}.
$$
\ee
Then $T$ is minimal, that is
$$
\left\|T\right\|=\left\|c_{i_0}(T)\right\|=\inf_{D\in \D(\bh)}\left\|T+D\right\|=\inf_{D\in \D(\kh)}\left\|T+D\right\|
$$
and moreover, $D=\text{Diag}\big(\{T_{nn}\}_{n\in\N}\big)$ is the unique bounded minimal diagonal operator for $T$.  
\end{teo} 
\section{The unitary Fredholm orbit of a Hermitian compact operator}
In this section we consider the unitary Fredholm orbit $\oa$ of a particular case of a Hermitian compact operator, that is: $A\in \kh$, $A=u{\rm Diag}\left(\{\lambda_i\}_{i\in\N}\right)u^*$, with $u\in \uc$ and $\{\lambda_i\}_{i\in\N}\subset \R$ such that $\lambda_i\neq \lambda_j$ for each $i\neq j$. Consider $\oa$ as defined in section \ref{preliminares} and $b={\rm Diag}\left(\{\lambda_i\}_{i\in\N}\right)\in \oa$. The isotropy $\I_b$ is the set $\{e^d:\ d\in \D(\kah)\}$ and $(T\oa)_b$ can be identified with the quotient space $\kah/\D(\kah)$. 
\begin{prop}
Let $b={\rm Diag}\left(\{\lambda_i\}_{i\in\N}\right)\in \oa$. For each $x\in (T\oa)_b$, if $Z\in \kah$ is such that $[Z,b]=x$, then
\begin{equation} \label{larga}
\|x\|_b=\inf_{D\in \D(\kah)}\left\|Z+D\right\|
\end{equation}
\end{prop}
\begin{proof}
If $Y_1, Y_2\in \{ Y\in\kah: [Y,b]=x\}$ then 
$$Y_1-Y_2\in\{D:\ [D,b]=Db-bD=0\}=\{b\}'$$
and since $b$ is a diagonal operator, then every $D$ is diagonal. Thus
$$Y_1-Y_2=D, \hbox{ with } D \hbox{ diagonal }$$
or equivalently: $Y_1=Y_2+D, \hbox{ with } D\in \D(\kah)$.
Then,
$$\|x\|_b=\inf\{\|Y\|: Y\in\kah \hbox{ such that } Y=Y_2+D,
\hbox{ with } D\in \D(\kah)\}.$$\end{proof}

Fix $x=[Z_r,b]=Z_rb-bZ_r\in \bb(\h)^{ah}$, where $Z_r$ is an anti-Hermitian operator defined as the infinite matrix given by
\begin{equation}
Z_r=i\begin{pmatrix}
0&r\g&r\g^2&r\g^3&\cdots\\
r\g&d_2&\g&\g^2&\cdots\\
r\g^2&\g&d_3&\g^2&\cdots\\
r\g^3&\g^2&\g^2&d_4&\cdots\\
\vdots&\vdots&\vdots&\vdots&\ddots\\
\end{pmatrix}\label{defi zr}
\end{equation}
\begin{equation}
=\underbrace{i\begin{pmatrix}
0&r\g&r\g^2&r\g^3&\cdots\\
r\g&0&\g&\g^2&\cdots\\
r\g^2&\g&0&\g^2&\cdots\\
r\g^3&\g^2&\g^2&0&\cdots\\
\vdots&\vdots&\vdots&\vdots&\ddots\\
\end{pmatrix}}_{Y_r}+\underbrace{i\begin{pmatrix}
0&0&0&0&\cdots\\
0&d_2&0&0&\cdots\\
0&0&d_3&0&\cdots\\
0&0&0&d_4&\cdots\\
\vdots&\vdots&\vdots&\vdots&\ddots\\
\end{pmatrix}}_{D_0}=Y_r+D_0.\label{defi Yr}
\end{equation}
The entries of the operator $Z_r$ are such that:
\be
\item $\g\in \R$ such that $\left|\g\right|<1$. 
\item For each $j\in\N$, $j>1$: $d_j=-\frac{1 -\g^{j-2}}{1-\g}-\frac{\g^j}{1-\g^2}$. Notice that $\lim \limits_{j \to \infty}d_j=\frac{1}{\g-1}$.
\item $r\geq \frac{\left\|Y^{[1]}+D_0\right\|}{\left(\sum_{k=1}^{\infty}\g^{2k}\right)^{1/2}}$, where 
$Y^{[1]}=Y_r-\begin{pmatrix}
0&r\g&r\g^2&r\g^3&\cdots\\
r\g&0&0&0&\cdots\\
r\g^2&0&0&0&\cdots\\
r\g^3&0&0&0&\cdots\\
\vdots&\vdots&\vdots&\vdots&\ddots\\
\end{pmatrix}.$
\ee
Observe that the definition of each $d_j$ is independent of the parameter $r$.

The operator $-iZ_r$ fulfills the conditions of minimality stated in Theorem \ref{teo minimal} stated in the Preliminaries and has been studied in \cite{bv}. Therefore, 
$$\left\|[Y_r]\right\|=\inf_{D\in \D(\kah)}\left\|Y_r+D\right\|=\left\|Y_r+D_0\right\|=\left\|Z_r\right\|.$$
Moreover, the diagonal operator $D_0$ is the unique minimal diagonal (bounded, but non compact) operator for $Y_r$. Since $D_0b-bD_0=0$, then $x=Y_rb-bY_r\in (T\oa)_b$ and
$$\left\|x\right\|_b=\left\|Z_rb-bZ_r\right\|_b=\inf_{D\in\D(\kh)}\ \left\|Y_r+D\right\|=\left\|\left[Y_r\right]\right\|=\left\|Z_r\right\|<\left\|Y_r+D\right\|$$
for all $D\in \D(\kah)$. In other words, there is no compact minimal lifting for $x$ in this case.

The following Proposition is a characterization of the unitary Fredholm group in terms of operators in $\kah$.
\begin{prop} \label{caracterizacion de unitarios}
$w\in \uc$ if and only if there exists $X\in \kah$ such that $w=e^X$.
\end{prop}
\begin{proof}
Given $w\in \uc$, by Lemma 2.1 in \cite{andruchow larotonda} there exists $X\in \kah$ such that $w=e^{X}$. On the other hand, consider $X\in\kah$ and the series expansion of $e^X$ 
$$e^{X}=1+X+\frac{1}{2}X^2+\frac{1}{3!}X^3+...$$
$$=1+X\left[1+\frac{1}{2}X+\frac{1}{3!}X^2+...\right]=1+K\ ,\ K\in \kk(\h),$$
and therefore $e^X\in \uc$.
\end{proof}
\begin{rem} \label{zr no esta en uc}
Even if $Z\notin \kah$, $e^{Z}$ may belong to $\uc$. Indeed, let $X_0\in \kah$, then $Z=X_0+2\pi i I\notin \kah$ but 
$$e^{X_0+2\pi i I}=e^{X_0}\in \uc.$$
\end{rem}
For $Z_r$ as in \eqref{defi zr} define the uniparametric curve $\s$ by
\begin{equation}
\s(t)=e^{tZ_r}be^{-tZ_r}\ ,\ t\in \left[-\frac{\pi}{2\left\|Z_r\right\|},\frac{\pi}{2\left\|Z_r\right\|}\right]. \label{curva minimal acotada}
\end{equation}
To prove that $\s$ is a curve in $\oa$, we introduce first the next result.
\begin{lem} \label{ exponenciales en uc}
Let $Z_r$ the operator defined in (\ref{defi zr}). Then for each $t\in \R$, there exist $z_t\in \C$, $\left|z_t\right|=1$ and $U(t)\in\uc$ such that
$$e^{tZ_r}=z_tU(t).$$
\end{lem}

\begin{proof}
Let $\alpha=-i\lim\limits_{n\to \infty} d_n=\frac{i}{1-\gamma}$. Then $e^{tZ_r+\alpha It}=e^{tZ_r}e^{t\alpha I}$.
Observe that $e^{t\alpha I}=e^{t\alpha}I$. Thus
$$e^{tZ_r}=e^{-t\alpha}e^{tZ_r+t\alpha I}=e^{-t\alpha}e^{tY_r+tD_0+t\alpha I},$$
with $e^{-t\alpha}\in \C$, $\left|e^{-t\alpha}\right|=1$ for every $t\in \R$. Moreover, $D_0+\alpha I\in \D(\kah)$, since it is a bounded diagonal and
$$\left|\left(D_0+\alpha I\right)_{jj}\right|=\left|-\frac{1 -\g^{j-2}}{1-\g}-\frac{\g^j}{1-\g^2}+\frac{1}{1-\gamma}\right|=\left|\frac{ \g^{j-2}}{1-\g}-\frac{\g^j}{1-\g^2}\right|\to 0$$
when $j\to\infty$.
Therefore, since $tZ_r+t\alpha I\in \kah$ for every $t\in \R$ then $U(t)=e^{tZ_r+t\alpha I}\in \uc$ and 
$$e^{tZ_r}=z_tU(t),\ {\rm with}\ z_t=e^{t\alpha}\in \C.$$
\end{proof}

\begin{rem} \label{oa in P} 
For any minimal lifting $Z\in\bah$ of $x=[Y,b]$, the curve $\kappa(t)=e^{Zt}be^{-Zt}$ has minimal length over all the smooth curves in $\PP=\{uAu^*:\ u\in\uu(\h)\}$ that join $\s(0)=b$ and $\s(t)$, with $\left|t\right|\leq \frac{\pi}{2\left\|Z_r\right\|}$ (Theorem II in \cite{dmr1}). Since $\oa\subseteq\PP$, then for each $t_0\in \left[-\frac{\pi}{2\left\|Z\right\|},\frac{\pi}{2\left\|Z\right\|}\right]$ follows that
$$\longi(\kappa)=\inf\{\longi(\chi):\chi\subset \PP,\chi\ {\rm is\ smooth},\chi(0)=b\ {\rm and}\ \chi(t_0)=\s(t_0)\}$$
$$\leq \inf\{\longi(\chi):\chi\subset \oa,\chi\ {\rm is\ smooth},\chi(0)=b\ {\rm and}\ \chi(t_0)=\s(t_0)\}$$
$$=\dist(b,\s(t_0)).$$
\end{rem}

Using the previous remark and Lemma \ref{ exponenciales en uc} we can prove the following Theorem.

\begin{teo} \label{teo curva minimal en oa}
Let $A=u{\rm Diag}\left(\{\lambda_i\}_{i\in\N}\right)u^*$, with $u\in \uc$ and $\{\lambda_i\}_{i\in\N}\subset \R$ such that $\lambda_i\neq\lambda_j$ for each $i\neq j$. Let $b={\rm Diag}\left(\{\lambda_i\}_{i\in\N}\right)\in \oa$ and the parametric curve $\s$ defined in (\ref{curva minimal acotada}). Then $\s$ satisfies:
\be
\item $\s(t)=e^{t(Z_r+\frac{i}{1-\g} I)}be^{-t(Z_r+\frac{i}{1-\g} I)}$, which means that $\s(t)\in \oa$ for every $t$.
\item $\s'(0)=x=Y_rb-bY_r=Z_rb-bZ_r\in (T\oa)_b$.
\item $\s$ has minimal length between all smooth curves in $\oa$ joining $b$ with $\s(t_0)$, for every $t_0\in\left[-\frac{\pi}{2\left\|Z_r\right\|},\frac{\pi}{2\left\|Z_r\right\|}\right]$. That is
$$\longi\left(\left.\s\right|_{[0,t_0]}\right)=\inf\{\longi(\chi):\chi\ {\rm is\ smooth},\chi(0)=b\ {\rm and}\ \chi(t_0)=\s(t_0)\}$$
$$=\dist(b,\s(t_0)).$$
\item $\longi\left(\left.\s\right|_{[0,t_0]}\right)=\left|t_0\right|\left\|x\right\|_b$, for each $t_0\in\left[-\frac{\pi}{2\left\|Z_r\right\|},\frac{\pi}{2\left\|Z_r\right\|}\right]$.
\ee
\end{teo}

\begin{proof}
\be
\item By Lemma \ref{ exponenciales en uc}, if $U(t)=e^{tZ_r+t\frac{i}{1-\g} I}$, then $\s$ can be rewritten as
$$\s(t)=z_tU(t)b(z_tU(t))^*=z_t\overline{z_t}U(t)bU^{-1}(t)$$
$$=U(t)bU^{-1}(t)=e^{t(Z_r+\frac{i}{1-\g} I)}be^{-t(Z_r+\frac{i}{1-\g} I)}$$
and $U(t)\in\uc$ for each $t\in \R$. Follows that $\s(t)\in \oa$ for every $t\in\R$. 
\item $\s'(0)=\left.e^{tZ_r}\left[Z_r,b\right]e^{-tZ_r}\right|_{t=0}$.
\item Observe that $\left\|Z_r\right\|=\left\|\left[Y_r\right]\right\|_{\bh^{ah}/\D(\bh)^{ah}}$ and $Z_r$ is (the unique) minimal lifting of $x=[Y_r,b]$ in $\bb(\h)$. Then, the result is a direct consequence of  Remark \ref{oa in P}.
\item Observe that $\longi(\s)=\int_{0}^{t_0}\left\|\s'(t)\right\|_{\s(t)}\ dt=t_0\left\|Y_rb-bY_r\right\|_b$. Indeed,
$$\left\|\s'(t)\right\|_{\s(t)}=\left\|Z_re^{tZ_r}be^{-tZ_r}-e^{tZ_r}bZ_re^{-tZ_r}\right\|_{\s(t)}=\left\|e^{tZ_r}\left[Z_r,b\right]e^{-tZ_r}\right\|_{\s(t)}$$
$$=\left\|z\overline{z}U(t)\left[Z_r,b\right]U^{-1}(t)\right\|_{\s(t)}=\left|z\right|^2\left\|U(t)\left[Z_r,b\right]U^{-1}(t)\right\|_{\s(t)}$$
$$=\left\|U(t)\left[Z_r,b\right]U^{-1}(t)\right\|_{U(t)bU^{-1}(t)}=\left\|Z_rb-bZ_r\right\|_b$$
$$=\left\|Y_rb-bY_r\right\|_b=\left\|x\right\|_b,$$
where the equality $\left\|U(t)\left[Z_r,b\right]U^{-1}(t)\right\|_{U(t)bU^{-1}(t)}=\left\|Z_rb-bZ_r\right\|_b$ holds due to the unitary invariance of the Finsler norm. 
\ee 
\end{proof}
Summarizing, if $Z_{\alpha}=Z_r+\frac{i}{1-\g} I\in \kah$, we obtained that the parametric curve given by
$$\pi_b\circ (e^{tZ_{\alpha}})=e^{tZ_{\alpha}}be^{-tZ_{\alpha}}$$
has minimal length between elements of $\oa$. Nevertheless, the operator $Z_{\alpha}$ is not a minimal element in its class (recall that $[Z_r]=\{Z_r+D:D\in \D(\kah)=[Y_r]\}$). On the other hand, 
$$e^{tZ_{\alpha}}be^{-tZ_{\alpha}}=e^{tZ_r}be^{-tZ_r}$$
and $Z_r$ is minimal, but it does not belong to $\kah$. We conclude with the following comment.
\begin{rem} \label{coro curvas minimales}
Let $b\in \oa$, $b={\rm Diag}\left(\{\lambda_i\}_{i\in\N}\right)$ such that $\lambda_i\neq\lambda_j$ for each $i\neq j$.
Then, there exist minimal length curves of the form $\rho(t)=e^{tZ}be^{-tZ}$ in $\oa$ such that they join $b$ with other points of the orbit, but with $Z\in \kah$ and $\left\|Z\right\|>\left\|\left[Z\right]\right\|_{\kah/\D(\kah)}$.
\end{rem}

\section{Approximation with minimal length curves of matrices} \label{seccion aproximaciones}
There are two main objectives in this section: the first is to build two sequences of minimal matrices which approximate $Z_r$ and $Z_r+\frac{i}{1-\g} I$  in the strong operator topology (SOT) and in the operator norm, respectively. The second objetive is to find a family of minimal length curves of matrices which approximates the curve $\s$ defined in (\ref{curva minimal acotada}).

Let $Y_r$ be the anti-Hermitian compact operator defined in (\ref{defi Yr}) and consider the following decomposition
\begin{equation}
Y_r=rL+Y^{[1]},\ {\rm where}\  L=i\begin{pmatrix}
0&\g&\g^2&\g^3&\cdots\\
\g&0&0&0&\cdots\\
\g^2&0&0&0&\cdots\\
\g^3&0&0&0&\cdots\\
\vdots&\vdots&\vdots&\vdots&\ddots\\
\end{pmatrix}.
\label{operador contraejemplo}
\end{equation}
 Let $D_0$ be the diagonal bounded operator defined in \eqref{defi Yr}. If $r\geq \frac{\left\|Y^{[1]}+D_0\right\|}{\left\|c_1(L)\right\|}$, then $Z_r=rL+Y^{[1]}+D_0$ is minimal.

Let us consider for each $n\in \N$ the orthogonal projection $P_n$ over the space generated by $\left\{e_1,...,e_n\right\}$. We define the following finite range operators
\begin{equation}
Y_n=r_nP_nLP_n+P_nY^{[1]}P_n, \label{defi Y_n}
\end{equation}
with $r_n\in\R_{>0}$ for each $n\in \N$. 
For each $n\in \N$ we define the diagonal operator $D_n=i{\rm Diag}\left(\{d_k^{(n)}\}_{k\in \N}\right)$ uniquely determined by the conditions:
\be
\item $d_1^{(n)}=0$;
\item $\left\langle c_1(Y_n+D_n),c_j(Y_n+D_n)\right\rangle=0$, for each $j\in\N$, $j\neq 1$;
\item $d_k^{(n)}=0$, for every $k>n$.
\ee
Thus, if we use the convention that $\sum_{j=0}^{2-3}\g^{j}=0=\sum_{j=n}^{n-1}\g^{2j-n+1}$ then each $d_k^{(n)}$ is determined for each $n$ as
\begin{equation}
\left.\begin{array}{c l r l}
d_k^{(n)}=-\sum_{j=0}^{k-3}\g^{j}-\sum_{j=i}^{n-1}\g^{2j-k}<0& \mbox{ if } &k\leq n\\
d_k^{(n)}=0& \mbox{ for all }& k>n.
\end{array}\right. \label{defi dn}
\end{equation}
The proof is by induction over the indices $k$ for every $n\in \N$. Observe that the choice of each $d_k^{(n)}$ is independent of the parameter $r_n$.

The following lemma will be used to prove the minimality of each $Y_n+D_n$ for a fixed $r_n$.

\begin{lem} \label{cota para norma de bloques}
Let $Y_n=r_nP_nLP_n+P_nY^{[1]}P_n$ and $D_n$ as defined in (\ref{defi Y_n}) and (\ref{defi dn}) for each $n\in \N$, respectively. Then
$$\sup_{n\in \N}\left\|P_nY^{[1]}P_n+D_n\right\|<\infty.$$
\end{lem}
\begin{proof}
Fix $n\in \N$. Since $\sup_{n\in \N}\left|d_n^{(n)}\right|\leq \left\|D_0\right\|$, for $D_0$ the diagonal operator defined in \eqref{defi Yr}, then
$$\left\|P_nY^{[1]}P_n+D_n\right\|\leq \left\|P_nY^{[1]}P_n\right\|+\left\|D_n\right\|\leq \left\|P_n\right\|^2\left\|Y^{[1]}\right\|+\sup_{1\leq k\leq n}\left|d_k^{(n)}\right|$$
$$\leq \left\|Y^{[1]}\right\|+\left|d_n^{(n)}\right| \leq \left\|Y^{[1]}\right\|+ \sup_{n\in \N}\left|d_n^{(n)}\right|\leq \left\|Y^{[1]}\right\|+\left\|D_0\right\|<\infty.$$
\end{proof}
As a consequence of this lemma, there exists a constant $M_0\in \R_{>0}$ such that:
\begin{equation}
M_0=\max\left\{\sup_{n\in \N}\left\|P_nY^{[1]}P_n+D_n\right\|,\left\|Y^{[1]}+D_0\right\|\right\}. \label{parametroalfa}
\end{equation}

Now we can prove the minimality of each $Y_n+D_n$ for all $n\in \N$.

\begin{prop} \label{norma minimales matriciales}
Let $Y_n=r_nP_nLP_n+P_nY^{[1]}P_n$ and $D_n$ as defined in (\ref{defi Y_n}) and (\ref{defi dn}) for each $n\in \N$, respectively. Consider the constant $M_0$ as in (\ref{parametroalfa}) and define $r_n=\frac{M_0}{\left\|c_1(P_nLP_n)\right\|}$. Then for each $n\in \N$ the operator $Y_n+D_n$ is minimal in $\kah/\D(\kah)$, that is
$$\left\|\left[Y_n\right]\right\|=\inf_{\tilde{D}\in \D(\kah)}\left\|Y_n+\tilde{D}\right\|=\left\|Y_n+D_n\right\|=M_0.$$
\end{prop}
\begin{proof}
Fix $n\in \N$. Without loss of generality, we can consider $Y_n+D_n$ as  an $n\times n$ matrix. Then 
\bit
\item $d_1^{(n)}=0$;
\item $\left\langle c_1(Y_n+D_n),c_j(Y_n+D_n)\right\rangle=0$, for each $j\in\N$, $2\leq j\leq n$;
\item $\left\|c_1(Y_n+D_n)\right\|=r_n\left\|c_1(P_nLP_n)\right\|=M_0\geq \left\|P_nY^{[1]}P_n+D_n\right\|$.
\eit
As an $n\times n$ matrix, $D_n$ is the unique minimal diagonal operator for $Y_n$ (see Theorem 8 in \cite{kv}). Since
$$\inf_{D\in\D(\kah)}\left\|Y_n+D\right\|=\min_{\tilde{D}\in \D(M_n(\C)^{ah})}\left\|Y_n+\tilde{D}\right\|,$$
follows that
$$\left\|\left[Y_n\right]\right\|=\left\|Y_n+D_n\right\|.$$
\end{proof}
Observe that the norm of the minimal operator $Y_n+D_n$ is $M_0$ for every  $n\in \N$.
\begin{rem}
For every $n\in \N$
$$\inf_{D\in\D(\kah)}\left\|Y_n+D\right\|=\min_{D'\in \D(M_n(\C)^{ah})}\left\|Y_n+D'\right\|=\left\|Y_n+D_n\right\|,$$
but there is no uniqueness of the $D'\in \D(\kah)$ that attain the minimum. Moreover, every block operator of the form 

$C_n=\begin{pmatrix}
D_n&0\\
0&D_c\\
\end{pmatrix}$, with $D_c$ diagonal and such that $\left\|D_c\right\|\leq\left\|c_1(Y_n)\right\|$ satisfies
$$\left\|Y_n+C_n\right\|=\max\left\{\left\|Y_n+D_n\right\|;\left\|D_c\right\|\right\}=\left\|Y_n+D_n\right\|=\left\|\left[Y_n\right]\right\|.$$
\end{rem}
Reconsider the operator $Y_r=rL+Y^{[1]}$ fixing $r=\frac{M_0}{\left\|c_1(L)\right\|}$. Note that
$$\frac{\left\|Y^{[1]}+D_0\right\|}{\left\|c_1(L)\right\|}\leq r<\infty$$
where the last inequality holds due to Lemma \ref{cota para norma de bloques}. Then, $Z_r=Y_r+D_0$ satisfies the hypothesis of Theorem \ref{teo minimal} and is a minimal operator with $D_0$, the unique (non compact) bounded diagonal operator such that 
$$\left\|[Y_r]\right\|=\inf_{D\in \D(\kah)}\left\|Y_r+D\right\|=\left\|Z_r\right\|.$$
Moreover, 
$$\left\|\left[Y_r\right]\right\|=\left\|c_1(Z_r)\right\|=\left\|c_1(Y_r)\right\|=M_0.$$
Therefore,
\begin{equation}
\left\|\left[Y_r\right]\right\|= \left\|\left[Y_n\right]\right\|,\ {\rm for\ all}\ n\in \N. \label{igualdad de infimos entre operador y matrices aproximantes}
\end{equation}
The following result relates $Y_r$ with $Y_n$.
\begin{prop} \label{convergencia Tn a T}
Let $Y_r$ be the operator defined in (\ref{operador contraejemplo}) and $\{Y_n\}_{n=1}^{\infty}$ the family of finite range operators defined in (\ref{defi Y_n}). If $M_0$ is the real constant defined in (\ref{parametroalfa}) such that $r=\frac{M_0}{\left\|c_1(L)\right\|}$ and  $r_n=\frac{M_0}{\left\|c_1(P_nLP_n)\right\|}$ for each $n\in \N$, are fixed. Then
\be
\item $\lim \limits_{n\to\infty}r_n=r$.
\item $Y_n\to Y_r$ when $n\to\infty$ in the operator norm.
\ee
\end{prop}

\begin{proof}
\be
\item Since $\left\|c_1(P_nLP_n)\right\|=\left(\sum_{i=1}^{n-1}\g^{2i}\right)^{\frac{1}{2}}$ and $\left\|c_1(L)\right\|=\big(\sum_{i=1}^{\infty}\g^{2i}\big)^{\frac{1}{2}}$, follows that $\lim \limits_{n\to\infty}r_n=r$.
\item $\left\|Y_r-Y_n\right\|=\left\|rL+Y^{[1]}-r_nP_nLP_n-P_nY^{[1]}P_n\right\|$
$$\leq \left\|rL\pm r_nL-r_nP_nLP_n\right\|+\left\|Y^{[1]}-P_nY^{[1]}P_n\right\|$$
$$\leq \left|r-r_n\right|\left\|L\right\|+\left|r_n\right|\left\|L-P_nLP_n\right\|+\left\|Y^{[1]}-P_nY^{[1]}P_n\right\|\to 0$$
\ee
when $n\to \infty$, since $L$ and $Y^{[1]}$ are Hilbert-Schmidt operators and $r_n\to r$.
\end{proof}

Observe that the numerical sequence $\{d_k^{(n)}\}_{n\in \N}$ defined in \eqref{defi dn} converges to $d_k$ when $n\to \infty$, for each $k\in\N$
$$d_k^{(n)}\searrow-\sum_{j=0}^{k-3}\g^{j}-\sum_{j=k}^{\infty}\g^{2j-k}=-\sum_{j=0}^{i-3}\g^{j}-\frac{\g^k}{1-\g^2}=d_k.$$
As a consequence, the sequence of diagonal operators $\{D_n\}_{n\in \N}$ converges SOT to the unique best approximant (non compact) diagonal $D_0\in \D(\bh)$ for $Y_r$. 

\begin{prop} \label{convergencia sot a D}
Let $Y_r$ be the operator defined in (\ref{defi Yr}) and $D_0$ the unique bounded diagonal operator such that $\left\|[Y_r]\right\|_{\kah/\D(\kah)}=\left\|Y_r+D_0\right\|$. Let $\{D_n\}_{n\in \N}$ be the sequence of finite range diagonal operators defined in (\ref{defi dn}). Then 
$$D_n\to D_0\ {\rm SOT\ when}\ n\to \infty.$$
\end{prop}

\begin{proof}
$\left\{D_n-D_0\right\}_{n\in \N}$ is a bounded family of $\bb(\h)$ and
$$\left(D_n-D_0\right)(e_k)=d_k^{(n)}-d_k\to 0$$
when $n\to \infty$ for every $e_k$ that belongs to the fixed orthonormal basis. Then standard arguments of operator theory imply that $D_n\to D_0$ SOT when $n\to\infty$ (see \cite{conway}). 
\end{proof}
Observe that Propositions \ref{convergencia Tn a T} and \ref{convergencia sot a D} imply that $\lim\limits_{n\to\infty}Y_n+D_n=Z_r$ SOT. Since $D_n\in \kah$ for all $n$ and $D_0\notin \D(\kah)$, the convergence can not be in the operator norm. To establish the second main result of this section we prove first the convergence in the operator norm of $Y_n+D_n\alpha I$ to $Z_r+\alpha I$, for a particular $\alpha \in \R$. 

\begin{prop} \label{convergencia en norma de no minimales}
Let $Y_r,D_0,\{Y_n\}_{n\in \N},\{D_n\}_{n\in \N},\{P_n\}_{n\in \N} $ be the operators and sequence of operators defined previously in \eqref{operador contraejemplo}, \eqref{defi Y_n} y \eqref{defi dn}. Then
$$Y_n+D_n+\frac{i}{1-\g} P_n\to Y_r+D_0+\frac{i}{1-\g} I,$$
in the operator norm when $n\to\infty$.
\end{prop}

\begin{proof}
Let $\epsilon>0$, then
$$\left\|Y_r+D_0+\frac{i}{1-\g} I-Y_n-D_n-\frac{i}{1-\g} P_n\right\| 
 \leq \left\|Y_r-Y_n\right\|+\left\|D+\frac{i}{1-\g} I-D_n-\frac{i}{1-\g} P_n\right\|.$$
By Proposition \ref{convergencia Tn a T}, there exists $n_1\in \N$ such that $\left\|Y_r-Y_n\right\|<\epsilon$, for all $n\geq n_1$.
Focus on the second term. For each $n\in \N$
$$\left\|D_0+\frac{i}{1-\g} I-D_n-\frac{i}{1-\g} P_n\right\|=\sup_{k\in \N}\left|d_k+\frac{1}{1-\g}-d_k^{(n)}-\left(\frac{1}{1-\g} P_n\right)_{kk}\right|$$
$$=\max\left\{\max\limits_{1\leq k\leq n} \left|\sum_{j=n}^{\infty}\g^{2j-k}\right|;\sup\limits_{k>n}\left|d_k+\frac{1}{1-\g}\right|\right\}.$$
By Proposition \ref{convergencia sot a D}, $\max\limits_{1\leq k\leq n} \left|\sum_{j=n}^{\infty}\g^{2j-k}\right|$ and $\sup\limits_{k>n}\left|d_k+\frac{1}{1-\g}\right|$ converges to $0$ when $n\to \infty$. Then, there exists $n_2\in \N$ such that for each $n\geq n_0$
$$\max\left\{\max\limits_{1\leq k\leq n} \left|\sum_{j=n}^{\infty}\g^{2j-k}\right|;\sup\limits_{k>n}\left|d_k+\frac{1}{1-\g}\right|\right\}<\epsilon.$$
Finally, if $n_0=\max\{n_1;n_2\}$ follows that
$$n\geq n_0\Rightarrow\ \left\|Y_r+D_0+\frac{i}{1-\g} I-Y_n-D_n-\frac{i}{1-\g} P_n\right\|<2\epsilon,$$
which means that $Y_n+D_n+\frac{i}{1-\g} P_n$ converges to $Y_r+D_0+\frac{i}{1-\g} I$ when $n\to\infty$ in the operator norm. \end{proof}

In the above proof we also obtained that $\{D_n+\frac{i}{1-\g} P_n\}_{n\in \N}$, which is a sequence of finite range operators, converges in the operator norm to $D_0+\frac{i}{1-\g} I\in \D(\kah)$. Even though $Y_n+D_n+\frac{i}{1-\g} P_n$ and $Y_r+D_0+\frac{i}{1-\g} I$ are not minimal operators, they are useful to construct minimal length curves in the unitary orbit of $A$. We will also use the operators $Y_n+D_n+\frac{i}{1-\g} P_n$ to construct a sequence of minimal length curves that converge to $\s$ defined in (\ref{curva minimal acotada}).

The first result in this direction is the convergence of the sequence of exponential curves in $\oa$.

\begin{prop} \label{convergencia de curvas}
Let $b\in \oa$ and $\varsigma_n(t)=e^{tZ_n}be^{-tZ_n}$, a sequence of curves in $\oa$ with $t\in \R$ and $\{Z_n\}_{n\in \N}\subset\kah$ such that $\left\|Z_n-Z\right\|\to 0$ when $n\to\infty$. If we define $\varsigma(t)=e^{tZ}be^{-tZ}$, then
$$\varsigma_n\to \varsigma$$
uniformly in the operator norm when $n\to\infty$ for any interval $[t_1,t_2]\subset \R$.
\end{prop}
\begin{proof}
Let $\epsilon>0$.
$$\left\|\varsigma_n(t)-\varsigma(t)\right\|\leq \left\|e^{tZ_n}be^{-tZ_n}-e^{tZ}be^{-tZ_n}\right\|+\left\|e^{tZ}be^{-tZ_n}-e^{tZ}be^{-tZ}\right\|$$
$$\leq \left\|\left(e^{tZ_n}-e^{tZ}\right)be^{-tZ_n}\right\|+\left\|e^{tZ}b\left(e^{-tZ_n}-e^{-tZ}\right)\right\|$$
$$\leq \left(\left\|e^{tZ_n}-e^{tZ}\right\|+\left\|e^{-tZ_n}-e^{-tZ}\right\|\right)\left\|b\right\|.$$
It is known that the exponential map $exp: \kah\to \uc$ is Lipschitz continuous in compact sets of $\kk(\h)$, then there exists $n_0\in \N$ such that
$${\rm for\ all}\ n\geq n_0\Rightarrow \ \left\{\begin{array}{c l r l}
\left\|e^{tZ_n}-e^{tZ}\right\|<\frac{\epsilon}{\left\|b\right\|}\\
\left\|e^{-tZ_n}-e^{-tZ}\right\|<\frac{\epsilon}{\left\|b\right\|},
\end{array}
\right.$$
for each $t$ in a closed interval $[t_1,t_2]\subset\R$. Therefore
$$\left\|\varsigma_n(t)-\varsigma(t)\right\|<\epsilon$$
for each $n\geq n_0$ and $t\in[t_1,t_2]$, which implies that $\varsigma_n\to \varsigma$ uniformly in the operator norm in that interval.
\end{proof}
If we consider the sequence $\{Y_n+D_n+\frac{i}{1-\g} P_n\}_{n\in N}$ and use Proposition \ref{convergencia en norma de no minimales} then 
$$Y_n+D_n+\frac{i}{1-\g} P_n\to Y_r+D_0+\frac{i}{1-\g} I$$
 in the operator norm when $n\to\infty$. Define for each $n\in\N$ and $t_0\in \R$ the curves parametrized by
\begin{equation}
\s_n(t)=e^{t(Y_n+D_n+\frac{i}{1-\g} P_n)}be^{-t(Y_n+D_n+\frac{i}{1-\g} P_n)},\ t\in[0,t_0]. \label{curvas minimales de rango finito}
\end{equation}
Observe that these can be considered as matricial type curves.
\begin{teo} \label{teo curvas minimales de matrices}
Let $A$ and $b\in \oa$ as in Theorem \ref{teo curva minimal en oa}. Let $\{\s_n\}_{n\in \N}$ be the sequence of curves defined in (\ref{curvas minimales de rango finito}), and $\s$ be the curve defined in (\ref{curva minimal acotada}). Then, for each $n\in \N$
\be
\item $\left\{\begin{array}{c l r l}
\s_n(0)=b\\
\s'_n(0)=Y_nb-bY_n\in (T\oa)_b.
\end{array}\right.$. 
\item $\s_n(t)=e^{t(Y_n+D_n)}be^{-t(Y_n+D_n)}$ for all $t$, since $\frac{i}{1-\g} P_n$ commutes with $Y_n+D_n$.
\item For each $t_0\in\left[-\frac{\pi}{2\left\|[Y_n]\right\|},\frac{\pi}{2\left\|[Y_n]\right\|}\right]=\left[-\frac{\pi}{2M_0},\frac{\pi}{2M_0}\right]$ holds that
$$\longi\left(\left.\s_n\right|_{[0,t_0]}\right)=\left|t_0\right|\left\|\left[Y_n\right]\right\|=\left|t_0\right|M_0=\longi\left(\left.\s\right|_{[0,t_0]}\right).$$
\item $\s_n:[0,t_0]\to \oa$ with $t_0\in \left[-\frac{\pi}{2M_0},\frac{\pi}{2M_0}\right]$ is a minimal length curve in $\oa$.
\item $\s_n'(0)\to \s'(0)$ in the norm  $\left\|.\right\|_b$ of $(T\oa)_b$.
\ee
Moreover, by Proposition \ref{convergencia de curvas}, $\s_n\to\s$ uniformly in the operator norm in the interval $\left[-\frac{\pi}{2M_0},\frac{\pi}{2M_0}\right]$.
\end{teo}
\begin{proof}
The proof of items (1), (2), (3) is analogous to the proof in Theorem \ref{teo curva minimal en oa}. The equality $\left\|[Y_n]\right\|=M_0=\left\|[Y_r]\right\|$ is due to Proposition \ref{norma minimales matriciales}. 

Since for each $n\in \N$ fixed $Y_n+D_n$ is a minimal compact operator, by Theorem I in \cite{dmr1} $\s_n$ is a minimal length curve between all curves in $\oa$ joining $\s_n(0)=b$ and $\s_n(t)$ with $\left|t\right|\leq \frac{\pi}{2\left\|Y_n+D_n\right\|}$. Then (4) is proved.

We proceed to prove (5): fix $\epsilon>0$. Then there exists $n_0\in \N$ such that if $n\geq n_0$ then $\left\|Y_n-Y_r\right\|<\epsilon$. Therefore,

$$\left\|\s_n'(0)-\s'(0)\right\|_b=\inf\left\{\left\|Z\right\|: Z\in \kah,\ [Z,b]=\left(Y_n-Y_r\right)b-b\left(Y_n-Y_r\right)\right\}$$
$$=\inf_{D\in \mathcal{D}\left(\kah\right)}\left\|Y_n-Y_r+D\right\|\leq\left\|Y_n-Y_r\right\|<\epsilon$$
for each $n\geq n_0$. Then $\left\|\s_n'(0)-\s'(0)\right\|_b\to 0$ when $n\to\infty$.

\end{proof}
Therefore, we obtained a minimal length curve $\s\subset \oa$ that can be uniformly approximated by minimal curves of matrices $\{\s_n\}$. Nevertheless, $\s$ does not have a minimal compact lifting, although each $\s_n$ has at least one minimal matricial lifting.

\section{Bounded minimal operators $Z+D$ with $Z\in \kk(\h)$ and non compact diagonal $D$}
Let $Y_r$, $D_0$ be the operators defined in (\ref{defi Yr}). To establish the equality $\s(t)=e^{Y_r+D_0+\frac{i}{1-\g} I}be^{-(Y_r+D_0+\frac{i}{1-\g} I)}$ in Theorem \ref{teo curva minimal en oa} the following properties  were essential: 
\be
\item $D_0+\frac{i}{1-\g} I\in \D(\kah)$ and
\item $\frac{i}{1-\g} I$ commutes with $Z_r$ and $b$ but $\frac{i}{1-\g} I\notin \kk(\h)$.
\ee
This can be generalized.
\begin{prop} \label{diagonal minimal con limite}
Let $Z\in\kah$ and suppose that there exists $D_1\in \D(\bh^{ah})$ such that
$$\left\|[Z]\right\|_{\kah/\D(\kah)}=\left\|Z+D_1\right\|$$
and $D_1$ is not compact. If there exists $\lambda\in i\R$ such that $\lim_{j\to\infty}\ (D_1)_{jj}=\lambda$, then the curve 
$$\chi(t)=e^{t(Z+D_1-\lambda I)}be^{-t(Z+D_1-\lambda I)}$$
has minimal length between all the smooth curves in $\oa$ joining $b$ with $\chi(t_0)$, for $t_0\in[\frac{\pi}{2\left\|Z\right\|}, \frac{\pi}{2\left\|[Z]\right\|}]$. 
\end{prop}
\begin{proof}
First observe that $Re\left((D_1)_{jj}\right)=0$ for each $j\in \N$, since $D_1\in \D(\bh^{ah})$. Then,
$$\lim_{j\to\infty}\ (D_1)_{j}=\lambda$$
and $\lambda\neq 0$ since $D_1$ is not compact.
Therefore, using functional calculus and Proposition 6 in \cite{bv}
$$\left\|Z+D_1-\lambda I\right\|=\max\{\left|-\left\|[Z]\right\|-\left|\lambda\right|\right|;\left\|[Z]\right\|-\left|\lambda\right|\}>\left\|[Z]\right\|.$$
Also $D_1-\lambda I\in \D(\kah)$, since $\left|(D_1-\lambda I)_{jj}\right|=\left|(D_1)_{jj}-\lambda\right|\to 0$, when $j\to \infty$. Then, $Z+D_1-\lambda I$ is not minimal in $\kah/\D(\kah)$ but the curve parameterized by
$$\chi(t)=e^{t(Z+D_1-\lambda I)}be^{-t(Z+D_1-\lambda I)}\in\oa$$
has minimal length, as $\chi$ is equal to the curve $\delta(t)=e^{t(Z+D_1)}be^{-t(Z+D_1)},$
which has minimal length in the homogeneous space $\{uAu^*:\ u\in \uu(\h)\}$ (Theorem II  in \cite{dmr1}). Therefore $\chi$ has minimal length in $\oa$.
\end{proof}
Given $Z\in \kah$, it is not true that every diagonal operator $D_1$ such that $Z+D_1$ is minimal fulfills the condition
$$\exists \lambda\in i\R\ \text{such that}\ \lim_{j\to\infty}\ (D_1)_{jj}=\lambda.$$
Indeed, consider the following operator
\begin{equation}
Z_0=i
\left(
\begin{array}{ccccccccccccc}
 0 & -\delta  & \gamma  & -\delta ^2 & \gamma ^2 & -\delta ^3 & \gamma ^3 &\cdots \\
 -\delta  & 0 & \gamma  & -\delta ^2 & \gamma ^2 & -\delta ^3 & \gamma ^3 &\cdots  \\
 \gamma  & \gamma  & 0 & -\delta ^2 & \gamma ^2 & -\delta ^3 & \gamma ^3 &\cdots  \\
 -\delta ^2 & -\delta ^2 & -\delta ^2 & 0 & \gamma ^2 & -\delta ^3 & \gamma ^3 &\cdots \\
 \gamma ^2 & \gamma ^2 & \gamma ^2 & \gamma ^2 & 0 & -\delta ^3 & \gamma ^3 &\cdots  \\
 -\delta ^3 & -\delta ^3 & -\delta ^3 & -\delta ^3 & -\delta ^3 & 0 & \gamma ^3&\cdots\\
 \gamma ^3 & \gamma ^3 & \gamma ^3 & \gamma ^3 & \gamma ^3 & \gamma ^3 & 0 &\cdots \\
 \vdots&\vdots&\vdots&\vdots&\vdots&\vdots&\vdots&\ddots
\end{array}
\right),\ {\rm with}\ \g,\delta\in (0,1). \label{defi Z0}
\end{equation}
It is easy to prove that $Z_0$ is a Hilbert Schmidt operator. 

Let $D'_0=i{\rm Diag}\left(\left\{d_n'\right\}_{n\in \N}\right)$ the unique bounded diagonal operator such that 
\begin{equation}
\left\langle c_1(Z_0),c_n(Z_0+D'_0)\right\rangle=0,\ \forall\ n\in \N. \label{condicion para D'}
\end{equation}
Simple calculations show that the condition (\ref{condicion para D'}) implies that $\{d_n'\}_{n\in\N}$ satisfies that
$$d_{2k}'=\left(\sum_{j=1}^{k-1}\de^j\right)-\left(\sum_{j=1}^{k-1}\g^j\right)+\frac{\de^{k+2}}{1-\de^2}+\left(\frac{\g^{2}}{\de}\right)^{k}\frac{1}{1-\g^2}$$
and
$$d_{2k-1}'=\left(\sum_{j=1}^{k-1}\de^j\right)-\left(\sum_{j=1}^{k-2}\g^j\right)-\frac{\g^{k+1}}{1-\g^2}-\left(\frac{\de^{2}}{\g}\right)^{k}\frac{\g}{1-\de^2}$$
for each $k\in \N$. If $\g^2\leq \de$ and $\de^2\leq \g$ both sequences,$\{d_{2k}'\}_{k\in\N}$ and $\{d_{2k-1}'\}_{k\in\N}$, are convergent.

If $Z_0^{[1]}$ is the operator $Z_0$ defined in \eqref{defi Z0} but with zeros in its first column and row and $r=\frac{\left\|Z_0^{[1]}+D'_0\right\|}{c_1(Z_0)}$ then $r(Z_0-Z_0^{[1]})+Z_0^{[1]}+D'_0$ is a minimal operator by Theorem \ref{teo minimal} and $D'_0$ is the unique bounded minimal diagonal operator for $r(Z_0-Z_0^{[1]})+Z_0^{[1]}$. Also, if we fix the conditions $\g^2=\de$ and $\de^2< \g$ then
$$\lim\limits_{k\to\infty}d_{2k}'=\frac{\de}{1-\de}-\frac{\g}{1-\g}+\frac{1}{1-\g^2}\ \text{and}\ \lim\limits_{k\to\infty}d_{2k-1}'=\frac{\de}{1-\de}-\frac{\g}{1-\g},$$
which implies that  $\{(D'_0)_{nn}\}_{n\in N}$ has no limit. We call these diagonals ``oscillant'' in the sense that the sequence $\{\left\langle De_n,e_n\right\rangle\}_{n\in \N}$ has at least two different limits). 

Observe that an approximation to $Z_0$ by matrices can be built as the one done in section \ref{seccion aproximaciones}. Consider for each $n\in \N$ the orthogonal projection $P_n$ and define the following finite range operators

\begin{equation}
Z_n=r_nP_n(Z_0-Z_0^{[1]})P_n+P_nZ_0^{[1]}P_n, \label{defi Z_n}
\end{equation}
with $r_n\in\R_{>0}$ for each $n\in \N$. 
For each $n\in \N$ we define a diagonal operator $D_n'=i{\rm Diag}\left(\{d_l^{(n)'}\}_{l\in \N}\right)$ uniquely determined for each $n$ as
\be
\item $d_1^{(n)'}=0$;
\item $\left\langle c_1(Z_n+D_n'),c_j(Z_n+D_n')\right\rangle=0$, for each $j\in\N$, $j\neq 1$;
\item $d_l^{(n)'}=0$, for every $l>n$.
\ee
Then, each $d_l^{(n)'}$ is determined for each $n$ as
\begin{equation}
\left.\begin{array}{c l r l}
d_{2k}^{(n)'}=\left(\sum_{j=1}^{k-1}\de^j\right)-\left(\sum_{j=1}^{k-1}\g^j\right)+\left(\sum_{j=1}^{\left\lfloor \frac{n-k}{2}\right\rfloor}\de^{2j+k}\right)+\frac{\sum_{j=1}^{\left\lfloor \frac{n-2}{2}\right\rfloor}\g^{2j}}{\de^{k}}&\mbox{if}\ k<\frac{n}{2}\\
d_{2k-1}^{(n)'}=\left(\sum_{j=1}^{k-1}\de^j\right)-\left(\sum_{j=1}^{k-2}\g^j\right)-\left(\sum_{j=k}^{\left\lfloor \frac{n-1}{2}\right\rfloor}\g^{2j-k+1}\right)-\frac{\sum_{j=1}^{\left\lfloor \frac{n}{2}\right\rfloor}\de^{2j}}{\g^{k-1}}&\mbox{if}\ k< \frac{n+1}{2}\\
d_l^{(n)'}=0&\mbox{for\ all}\ l>n.
\end{array}\right.  \label{defi dlbis}
\end{equation}
The proof is by induction over the indices $k$ for every $n\in \N$. Observe the folllowing:
\be
\item The choice of each $d_k^{(n)'}$ is independent from the parameter $r_n$.
\item $\lim\limits_{n\to\infty}d_{2k}^{(n)'}=\left(\sum_{j=1}^{k-1}\de^j\right)-\left(\sum_{j=1}^{k-1}\g^j\right)+\left(\sum_{j=1}^{\infty}\de^{2j+k}\right)+\frac{\sum_{j=1}^{\infty}\g^{2j}}{\de^{k}}=d_{2k}'$.
\item $\lim\limits_{n\to\infty}d_{2k-1}^{(n)'}=\left(\sum_{j=1}^{k-1}\de^j\right)-\left(\sum_{j=1}^{k-2}\g^j\right)-\left(\sum_{j=k}^{\infty}\g^{2j-k+1}\right)-\frac{\sum_{j=1}^{\infty}\de^{2j}}{\g^{k-1}}$

$=d_{2k-1}'$.
\item For every $k\in \N$ and for each $n$: 
$$d_{2k-1}'\leq d_{2k-1}^{(n)'}\leq d_{2k}^{(n)'}\leq d_{2k}'.$$
Then, $\left\|D'_0\right\|=\sup\limits_{k\in \N}\{\left|d_{2k-1}'\right|;\left|d_{2k}'\right|\}\geq \left\|D'_n\right\|$.
\item $D_n'\to D'_0$ SOT, since $\text{Diag}\left(\{d_{2k}^{(n)'}\}_{k\in \N}\right)\to \text{Diag}\left(\{d_{2k}'\}_{k\in \N}\right)$ SOT and  $\text{Diag}\left( \{d_{2k-1}^{(n)'}\}_{k\in \N}\right)\to \text{Diag}\left( \{d_{2k-1}'\}_{k\in \N}\right)$ SOT.
\ee

With the previous properties, there exists $M_1\in \R_{>0}$ such that:
\begin{equation}
M_1=\max\left\{\sup_{n\in \N}\left\|P_nZ_0^{[1]}P_n+D_n'\right\|,\left\|Z_0^{[1]}+D'_0\right\|\right\}. \label{parametrom1}
\end{equation}
For any injective $\sigma:\N\to\N$ define the projection
\begin{equation}
P^{\sigma}=\sum_{k\in \N}e_{\sigma(k)}\otimes e_{\sigma(k)}. \label{psigma}
\end{equation}
Thus, the following result is a direct consequence of all previous remarks.
\begin{teo} \label{convergencia oscilante doble}
Let $Z_0$, $D_0'$, $Z_n=r_nP_n(Z_0-Z_0^{[1]})P_n+P_nZ_0^{[1]}P_n$ and $D'_n$ be the operators defined in \eqref{defi Z0}, \eqref{condicion para D'}, \eqref{defi Z_n} and \eqref{defi dlbis} for each $n\in \N$, respectively. Consider the real constant $M_1$ as in \eqref{parametrom1} and define $r_n=\frac{M_1}{\left\|c_1\left(P_n(Z_0-Z_0^{[1]})P_n\right)\right\|}$ for each $n\in \N$ and $r=\frac{M_1}{\left\|c_1\left(Z_0-Z_0^{[1]}\right)\right\|}$. If
$$\lambda=\lim\limits_{n\to\infty}d_{2k}',\ \mu=\lim\limits_{n\to\infty}d_{2k-1}',$$
then
\be
\item $Z_n+D'_n$ is minimal in $\kah/\D(\kah)$ and 
$$\left\|\left[Z_n\right]\right\|=\inf_{D\in \D(\kah)}\left\|Z_n+D\right\|=\left\|Z_n+D'_n\right\|=M_1.$$
\item If $P^{\sigma_1}$ and $P^{\sigma_2}$ are the projections defined in \eqref{psigma} for $\sigma_1(k)=2k$ and $\sigma_2(k)=2k-1$, respectively, then
$$Z_n+D'_n-\lambda P_nP^{\sigma_1}P_n-\mu P_nP^{\sigma_2}P_n\to r(Z_0-Z_0^{[1]})+Z_0^{[1]}+D'_0-\lambda P^{\sigma_1}-\mu P^{\sigma_2}$$
in the operator norm when $n\to\infty$.
\ee
\end{teo}

\begin{proof}
\be
\item Observe that if $D'_n$ is determined as in \eqref{defi dlbis} the operator

$-i(Z_n+D_n')$ fullfils the conditions stated in Theorem \ref{teo minimal} and
$$\inf_{D\in \D(\kah)}\left\|Z_n+D\right\|=\left\|Z_n+D'_n\right\|=\left\|c_1(Z_n+D'_n)\right\|=M_1.$$
\item Let $\epsilon>0$. Since $Z_0$ is compact and $r_n\to r$ then there exists $n_1\in \N$ such that
$$\left\|Z_n-r(Z_0-Z_0^{[1]})+Z_0^{[1]}\right\|<\frac{\epsilon}{2},$$
for each $n\geq n_1$.
Similarly than in the case of diagonal with one limit point (see proof of Proposition \ref{convergencia en norma de no minimales}), for each $n\in \N$:
$$\left\|D'_n-\lambda P_nP^{\sigma_1}P_n-\mu P_nP^{\sigma_2}P_n-D'_0+\lambda P^{\sigma_1}+\mu P^{\sigma_2}\right\|$$
$$=\sup\limits_{l\in \N}\left|d'_l-\lambda \left(P_nP^{\sigma_1}P_n\right)_{ll}-\mu \left(P_nP^{\sigma_2}P_n\right)_{ll}-d_l^{(n)'}-\lambda \left(P^{\sigma_1}\right)_{ll}-\mu \left(P^{\sigma_2}\right)_{ll}\right|$$
\begin{equation}
=\max\left\{\max\limits_{1\leq l\leq n} \left|d'_l-d_l^{(n)'}\right|;\sup\limits_{k>n}\left|d'_{2k}-\lambda \right|;\sup\limits_{k>n}\left|d'_{2k-1}-\mu \right|\right\}. \label{norma diagonal dos limites}
\end{equation}
Since $\lim\limits_{n\to\infty}d_{2k}^{(n)'}=d_{2k}'$, $\lim\limits_{n\to\infty}d_{2k-1}^{(n)'}=d_{2k-1}'$, $\lim\limits_{n\to\infty}d_{2k}'=\lambda$ and $\lim\limits_{n\to\infty}d_{2k-1}'=\mu$, there exists $n_2\in \N$ such that the last expression is smaller than $\frac{\epsilon}{2}$ for every $n\geq n_2$.
Therefore, it holds that
$$\left\|Z_n+D'_n-\lambda P_nP^{\sigma_1}P_n-\mu P_nP^{\sigma_2}P_n 
- \left[r(Z_0-Z_0^{[1]})+Z_0^{[1]}+D'_0-\lambda P^{\sigma_1}-\mu P^{\sigma_2}\right]\right\|$$
$$
\leq \left\|Z_n-r(Z_0-Z_0^{[1]})+Z_0^{[1]}\right\|
+\left\|D'_n-\lambda P_nP^{\sigma_1}P_n-\mu P_nP^{\sigma_2}P_n-D'_0+\lambda P^{\sigma_1}+\mu P^{\sigma_2}\right\|<\epsilon
$$
\ee  
for every $n\geq \max\{n_1;n_2\}$. \end{proof}
\begin{rem}
As $r(Z_0-Z_0^{[1]})+Z_0^{[1]}$, with $Z_0$ and $r$ defined previously, there exist other compact operators such that its best bounded diagonal approximant oscillates. Moreover, there exist examples of minimal bounded operators in which the elements on the main diagonal are the union of $m$ subsequences ($m\in\N$) such that each one converges to a different limit. For those $m-$oscillant operators an analogous result as that of Theorem \ref{convergencia oscilante doble} can be obtained with essentially the same arguments.
Nevertheless, the techniques used in Theorems \ref{teo curva minimal en oa} and \ref{teo curvas minimales de matrices} to prove that the curves constructed in \eqref{curva minimal acotada} and \eqref{curvas minimales de rango finito} belong to $\oa$ cannot be adapted to the case of an oscillant minimal diagonal for a compact $Z\in \kah$.
\end{rem}

\end{document}